\RequirePackage{fix-cm}
\documentclass[smallcondensed,envcountsame]{svjour3} 
\smartqed  
\usepackage{graphicx, amssymb}
\usepackage{times,a4wide,mathrsfs}
\usepackage{amsmath}
\usepackage{amsfonts}

\newcommand{\C}{\mathbb{C}}

\newcommand{\QQ}{\mathbb{Q}}
\newcommand{\NN}{\mathbb{N}}
\newcommand{\PP}{\mathbb{P}}
\newcommand{\LLL}{\mathbb{L}}

\newcommand{\MM}{\mathcal M}

\newcommand{\coker}{\hbox{Coker\,}}

\newcommand{\ima}{\hbox{Im}}
\newcommand{\rom}{\romannumeral}

\makeatletter
\newcommand*{\da@rightarrow}{\mathchar"0\hexnumber@\symAMSa 4B }
\newcommand*{\da@leftarrow}{\mathchar"0\hexnumber@\symAMSa 4C }
\newcommand*{\xdashrightarrow}[2][]{%
  \mathrel{%
    \mathpalette{\da@xarrow{#1}{#2}{}\da@rightarrow{\,}{}}{}%
  }%
}
\newcommand{\xdashleftarrow}[2][]{%
  \mathrel{%
    \mathpalette{\da@xarrow{#1}{#2}\da@leftarrow{}{}{\,}}{}%
  }%
}
\newcommand*{\da@xarrow}[7]{%
  \sbox0{$\ifx#7\scriptstyle\scriptscriptstyle\else\scriptstyle\fi#5#1#6\m@th$}%
  \sbox2{$\ifx#7\scriptstyle\scriptscriptstyle\else\scriptstyle\fi#5#2#6\m@th$}%
  \sbox4{$#7\dabar@\m@th$}%
  \dimen@=\wd0 %
  \ifdim\wd2 >\dimen@
    \dimen@=\wd2 %
  \fi
  \count@=2 %
  \def\da@bars{\dabar@\dabar@}%
  \@whiledim\count@\wd4<\dimen@\do{%
    \advance\count@\@ne
    \expandafter\def\expandafter\da@bars\expandafter{%
      \da@bars
      \dabar@ 
    }%
  }%
  \mathrel{#3}%
  \mathrel{%
    \mathop{\da@bars}\limits
    \ifx\\#1\\%
    \else
      _{\copy0}%
    \fi
    \ifx\\#2\\%
    \else
      ^{\copy2}%
    \fi
  }%
  \mathrel{#4}%
}
\makeatother

\newtheorem{convention}{Conventions}

\newtheorem{nonumbering}{Theorem}

\newtheorem{sublemma}[theorem]{Sublemma}

\newtheorem{resume}{R\'esum\'e}

 \journalname{}

\begin{document}

\title{A remark on the motive of the Fano variety of lines of a cubic}

\author{Robert Laterveer}

\institute{CNRS - IRMA, Universit\'e de Strasbourg \at
              7 rue Ren\'e Descartes \\
              67084 Strasbourg cedex\\
              France\\
              \email{laterv@math.unistra.fr}   }

\date{Received: date / Accepted: date}

\maketitle

\begin{abstract} Let $X$ be a smooth cubic hypersurface, and let $F$ be the Fano variety of lines on $X$. We establish a relation between the Chow motives of $X$ and $F$.
This relation implies in particular that if $X$ has finite--dimensional motive (in the sense of Kimura), then $F$ also has finite--dimensional motive. This proves finite--dimensionality for motives of Fano varieties of cubics of dimension $3$ and $5$, and of certain cubics in other dimensions. 
\end{abstract}

\begin{resume}\normalfont Soit $X$ une hypersurface cubique lisse, et soit $F$ la vari\'et\'e de Fano param\'etrant les droites contenues dans $X$. On \'etablit une relation entre les motifs de Chow de $X$ et de $F$. 
Cette relation implique le fait que $F$ a motif de dimension finie (au sens de Kimura) \`a condition que $X$ a motif de dimension finie. En particulier, si $X$ est une cubique lisse de dimension $3$ ou $5$, alors $F$ a motif de dimension finie.
\end{resume}

\keywords{Algebraic cycles \and Chow groups \and motives \and finite--dimensional motives \and cubics \and Fano variety of lines}

\subclass{14C15, 14C25, 14C30, 14J70, 14N25}

\section{Introduction}

The notion of finite--dimensional motive, developed independently by Kimura and O'Sullivan \cite{Kim}, \cite{An}, \cite{MNP}, \cite{J4}, \cite{Iv} has given important new impetus to the study of algebraic cycles. To give but one example: thanks to this notion, we now know the Bloch conjecture is true for surfaces of geometric genus zero that are rationally dominated by a product of curves \cite{Kim}. It thus seems worthwhile to find concrete examples of varieties that have finite--dimensional motive, this being (at present) one of the sole means of arriving at a satisfactory understanding of Chow groups. 

The present note aims to contribute something to the list of examples of varieties with finite--dimensional motive, by considering Fano varieties of lines of smooth cubics 
over $\C$. The main result is as follows:

\begin{nonumbering}[=theorem \ref{main}]  Let $X\subset\PP^{n+1}(\C)$ be a smooth cubic hypersurface, and let $F(X)$ denote the Fano variety of lines on $X$. If $X$ has finite--dimensional motive, then also $F(X)$ has finite--dimensional motive.
\end{nonumbering}

In particular, this implies that for smooth cubics $X$ of dimension $3$ or $5$, the Fano variety $F(X)$ has finite--dimensional motive. In the first case, the dimension of $F(X)$ is $2$, while in the second case it is $6$. The case $n=3$ is also proven (in a different way) in \cite{Diaz}. Some more examples where theorem \ref{main} applies are given in corollary \ref{ex}. 

Theorem \ref{main} follows from a more general result. This more general result relates the Chow motives of $X$ and $F=F(X)$ for any smooth cubic:

\begin{nonumbering}[=theorem \ref{main2}] Let $X\subset\PP^{n+1}(\C)$ be a smooth cubic hypersurface. Let $F:=F(X)$ denote the Fano variety of lines on $X$, and let $X^{[2]}$ denote the second Hilbert scheme of $X$. There is an isomorphism of Chow motives
  \[ h(F)(2)\oplus\bigoplus_{i=0}^n h(X)(i)\cong h(X^{[2]})\ \ \ \hbox{in}\ \MM_{\rm rat}\ .\]
\end{nonumbering}

This relation of Chow motives is inspired by (and formally similar to) a relation between $X$ and $F$ in the Grothendieck ring of varieties that was discovered by Galkin--Shinder \cite{GS} (cf. remark \ref{k0}).

\vskip0.4cm

\begin{convention} All varieties will be projective irreducible varieties over $\C$.

All Chow groups will be with rational coefficients: for $X$ smooth of dimension $n$, we will write $A^j(X)=A_{n-j}(X)$ for the Chow group 
of codimension $j$ cycles with $\QQ$--coefficients modulo rational equivalence. We will write $A^j_{hom}(X)$ and $A^j_{AJ}(X)$ for the subgroups of homologically trivial (resp. Abel--Jacobi trivial) cycles.

The category $\MM_{\rm rat}$  will denote the contravariant category of pure motives with respect to rational equivalence, as in \cite{Sch}, \cite{MNP}. For a morphism $f\colon X\to Y$ between smooth varieties, we will write $\Gamma_f\in A^{\dim Y}(X\times Y)$ for the graph of $f$.
\end{convention}

\section{Finite--dimensionality}

We refer to \cite{Kim}, \cite{An}, \cite{MNP}, \cite{Iv}, \cite{J4} for basics on the notion of finite--dimensional motive. 
An essential property of varieties with finite--dimensional motive is embodied by the nilpotence theorem:

\begin{theorem}[Kimura \cite{Kim}]\label{nilp} Let $X$ be a smooth projective variety of dimension $n$ with finite--dimensional motive. Let $\Gamma\in A^n(X\times X)_{}$ be a correspondence which is numerically trivial. Then there is $N\in\NN$ such that
     \[ \Gamma^{\circ N}=0\ \ \ \ \in A^n(X\times X)_{}\ .\]
\end{theorem}

 Actually, the nilpotence property (for all powers of $X$) could serve as an alternative definition of finite--dimensional motive, as shown by a result of Jannsen \cite[Corollary 3.9]{J4}.
   Conjecturally, all smooth projective varieties have finite--dimensional motive \cite{Kim}. We are still far from knowing this, but at least there are quite a few non--trivial examples:
 
\begin{remark} 
The following varieties have finite--dimensional motive: abelian varieties, varieties dominated by products of curves \cite{Kim}, $K3$ surfaces with Picard number $19$ or $20$ \cite{P}, surfaces not of general type with $p_g=0$ \cite[Theorem 2.11]{GP}, certain surfaces of general type with $p_g=0$ \cite{GP}, \cite{PW}, \cite{V8}, Hilbert schemes of surfaces known to have finite--dimensional motive \cite{CM}, generalized Kummer varieties \cite[Remark 2.9(\rom2)]{Xu},
 threefolds with nef tangent bundle \cite{Iy} (an alternative proof is given in \cite[Example 3.16]{V3}), fourfolds with nef tangent bundle \cite{Iy2}, log--homogeneous varieties in the sense of \cite{Br} (this follows from \cite[Theorem 4.4]{Iy2}), certain threefolds of general type \cite[Section 8]{V5}, varieties of dimension $\le 3$ rationally dominated by products of curves \cite[Example 3.15]{V3}, varieties $X$ with $A^i_{AJ}(X)_{}=0$ for all $i$ \cite[Theorem 4]{V2}, products of varieties with finite--dimensional motive \cite{Kim}.
\end{remark}

\begin{remark}
It is an embarassing fact that up till now, all examples of finite-dimensional motives happen to lie in the tensor subcategory generated by Chow motives of curves, i.e. they are ``motives of abelian type'' in the sense of \cite{V3}. On the other hand, there exist many motives that lie outside this subcategory, e.g. the motive of a very general quintic hypersurface in $\PP^3$ \cite[7.6]{Del}.
\end{remark}

\section{Main theorem}

This section contains the proof of the main result of this note, as announced in the introduction:

\begin{theorem}\label{main} Let $X\subset\PP^{n+1}(\C)$ be a smooth cubic hypersurface, and let $F:=F(X)$ denote the Fano variety of lines on $X$. If $X$ has finite--dimensional 
motive (resp. motive of abelian type), then $F$ has finite--dimensional motive (resp. motive of abelian type).
\end{theorem}

Theorem \ref{main} follows from a more general result:

\begin{theorem}\label{main2} Let $X\subset\PP^{n+1}(\C)$ be a smooth cubic hypersurface. Let $F:=F(X)$ denote the Fano variety of lines on $X$, and let $X^{[2]}$ denote the second Hilbert scheme of $X$. There is an isomorphism of Chow motives
  \[ h(F)(2)\oplus\bigoplus_{i=0}^n h(X)(i)\cong h(X^{[2]})\ \ \ \hbox{in}\ \MM_{\rm rat}\ .\]
  \end{theorem}

\begin{proof} The argument hinges on the following geometric relation between $X$ and $F$, which is specific to cubics:

\begin{proposition}[Galkin--Shinder \cite{GS}, Voisin \cite{V15}]\label{gs} Let $X\subset\PP^{n+1}(\C)$ be a smooth cubic hypersurface, and let $X^{[2]}$ denote its second Hilbert scheme. There exists a birational map
  \[  \phi\colon\ \ X^{[2]}\ \dashrightarrow\ W\ ,\]
  where $W$ is a $\PP^n$--bundle over $X$. The map $\phi$ admits a resolution of indeterminacy
  \[ \begin{array}[c]{ccc}
     & Y &\\
   {}^{\phi_1} \swarrow && \searrow {}^{\phi_2}\\
    X^{[2]}\ \ \ \ \ & \xdashrightarrow{\phi} & \ \ \ \ \ \ W\ .\\
    \end{array}\]
   Here the morphism
$\phi_1\colon Y\to X^{[2]}$ is the blow--up with center $\tau\colon Z\subset X^{[2]}$ of codimension $2$, and $Z$ has the structure of a $\PP^2$--bundle $p\colon Z\to F$. The morphism $\phi_2\colon Y\to W$ is the blow--up with center $\tau^\prime\colon Z^\prime\subset W$ of codimension $3$, and $Z^\prime$ has the structure of a $\PP^1$--bundle $p^\prime\colon Z^\prime\to F$.

Moreover, the diagram
  \[ \begin{array} [c]{ccc}
  & E &\\
   {}^{f} \swarrow && \searrow {}^{f^\prime}\\
    Z\ \ \ \ \ \ & & \ \ \ \ \ \ Z^\prime\ \\
     {}_{p} \searrow && \swarrow {}_{p^\prime}\\
     & F & \\
     \end{array}\]
     commutes, where $E\subset Y$ denotes the exceptional divisor of $\phi_1$ and $\phi_2$, and $f$ (resp. $f^\prime$) denotes the restriction of $\phi_1$ (resp. $\phi_2$) to $E$.
\end{proposition}

\begin{proof} The map $\phi$ is defined in \cite[Proof of Theorem 5.1]{GS}. The existence of the variety $Y$ with two different blow--up structures as indicated is \cite[Proposition 2.9]{V15}.

For the ``Moreover'' part, we inspect the proof of  \cite[Proposition 2.9]{V15}. This proof contains an explicit description of the exceptional divisor $E$ (denoted $Q_{P_2}$ in loc. cit.):
  \[ E =\Bigl\{  (u, x+y, [\ell])\ \vert\ \ell\subset X,\ x+y\in \ell^{[2]},\ u\in\ell\Bigr\}\ ,\]
  where the pair $x+y$ is in $X^{[2]}$ and $\ell$ denotes a line.
  The morphism $f$ sends a triple $(u,x+y,[\ell])$ to the pair $(x+y,[\ell])$. The image $f(E)$ is the locus of length $2$ subschemes $x+y\in X^{[2]}$ contained in a line $\ell$. Thus, $f(E)$ identifies with $Z$ (denoted $P_2$ in loc. cit.), and $p\circ f$ sends $(u,x+y,[\ell])$ to $[\ell]\in F$. The morphism $f^\prime$ (which is $\widetilde{\Phi}$ restricted to $Q_{P_2}$ in the notation of loc. cit.) sends a triple $ (u,x+y,[\ell])$ to the pair $(u,[\ell])$. The image $Z^\prime=f^\prime(E)$ (which is denoted $P$ in loc. cit.) has a $\PP^1$--bundle structure $p\colon Z^\prime\to F$ obtained by sending $(u,[\ell])$ to $[\ell]\in F$. This proves the ``Moreover'' assertion of proposition \ref{gs}.
\end{proof}

We now proceed with the proof of theorem \ref{main2}. As is well--known, a birational map $\phi\colon X^{[2]}\dashrightarrow W$ induces
homomorphisms
  \[  \begin{split} &\phi_\ast\colon\ \   A^j(X^{[2]})\ \to\ A^j(W)\ ,\\
                          & \phi^\ast\colon\ \   A^j(W)\ \to\ A^j(X^{[2]})\ ,\\
                          \end{split}\]
                          defined by the correspondence $\bar{\Gamma_\phi}$ (the closure of the graph of $\phi$) resp. its transpose.
  As a first step, we relate $F$ and $X^{[2]}$ on the level of Chow groups:

\begin{proposition}\label{step1} Let $X\subset\PP^{n+1}(\C)$ be a smooth cubic hypersurface, and let $F=F(X)$ be its Fano variety of lines. The map
  \[   \begin{split} A^{j-2}(F)\oplus A^j(W)\ &\to\ A^j(X^{[2]})\ ,\\ 
               (a,b)\ &\mapsto\  \tau_\ast  p^\ast(a)+  \phi^\ast(b)\\
               \end{split}  \]
  is an isomorphism for all $j$.
  \end{proposition}
  
  \begin{proof}
 It will be convenient to prove proposition \ref{step1} in a more abstract set--up. That is, we forget for the time being that we are dealing with cubics and Fano varieties and we only keep the geometric structure provided by proposition \ref{gs}. In this abstract set--up, we will prove the isomorphism of proposition \ref{step1}:
  
 \begin{proposition}\label{step1a} Let $V$ and $V^\prime$ be smooth projective varieties of dimension $m$. Assume there is a birational map 
   \[ \phi\colon\ \ V\ \dashrightarrow\ V^\prime\ ,\]
   and a
 commutative diagram
      \[ \begin{array}[c]{ccc}
     & Y &\\
   {}^{\phi_1} \swarrow && \searrow {}^{\phi_2}\\
    V \ \ \ \ \ & \xdashrightarrow{\phi} & \ \ \ \ \ \ V^\prime\ ,\\
    \end{array}\]
  where $\phi_1$ is the blow--up with smooth codimension $2$ center $\tau\colon Z\subset V$, and $\phi_2$ is the blow--up with smooth codimension $3$ center $\tau^\prime\colon Z^\prime\subset V^\prime$. Assume moreover there is a commutative diagram
      \[ \begin{array} [c]{ccc}
  & E &\\
   {}^{f} \swarrow && \searrow {}^{f^\prime}\\
    Z\ \ \ \ \ \ & & \ \ \ \ \ \ Z^\prime\ \\
     {}_{p} \searrow && \swarrow {}_{p^\prime}\\
     & F & \\
     \end{array}\]
      where $E$ denotes the exceptional divisor of $\phi_1$ and $\phi_2$, and $f$ (resp. $f^\prime$) denotes the restriction of $\phi_1$ (resp. $\phi_2$) to $E$, and $p$ (resp. $p^\prime$) is a $\PP^2$--bundle (resp. $\PP^1$--bundle) over a smooth projective variety $F$. Then the map
         \[   \begin{split} A^{j-2}(F)\oplus A^j(V^\prime)\ &\to\ A^j(V)\ ,\\ 
               (a,b)\ &\mapsto\  \tau_\ast  p^\ast(a)+  \phi^\ast(b)\\
               \end{split}  \]  
               is an isomorphism for all $j$.  
  \end{proposition}
  
 Proposition \ref{step1} is then the conjunction of propositions \ref{gs} and \ref{step1a}. We now prove proposition \ref{step1a}.
  
 For any $j$, there is a diagram with split--exact rows
  \begin{equation}\label{diag0}  \begin{array}[c]{ccccccc}
     0\to & A^{j-2}(Z) & \xrightarrow{\alpha} & A^{j-1}(E)\oplus A^j(V) & \xrightarrow{\beta} & A^j(Y) & \to 0\\
          & & &\ \ \ \ \ \ \ \ \ \  \uparrow{\scriptstyle (\psi,\phi^\ast)} &&\ \ \ \  \ \ \ \uparrow{\scriptstyle \tau } & \\
          0\to & A^{j-3}(Z^\prime) & \xrightarrow{\alpha^\prime} & A^{j-1}(E)\oplus A^j(V^\prime) & \xrightarrow{\beta^\prime} & A^j(Y) & \to 0\\
          \end{array}\end{equation}
    Here, the arrow labelled $\alpha$ is defined as $(c_1(G)\cdot f^\ast()  ,\tau_\ast)$ where $G$ is the excess normal bundle of the embedding $Z\subset V$ (as defined in \cite[Section 6.7]{F}). A left--inverse to $\alpha$ is given by $(a,b)\mapsto f_\ast(a)$. 
    The arrow labelled $\beta$ is defined as $i_\ast-(\phi_1)^\ast$, where $i\colon E\to Y$ denotes the inclusion morphism.
    The arrow labelled $\alpha^\prime$ is defined as
    $( c_2(G^\prime)\cdot (f^\prime)^\ast(), (\tau^\prime)_\ast)$, where $G^\prime$ is the excess normal bundle of the embedding $Z^\prime\subset V^\prime$. A left--inverse to $\alpha^\prime$ is given by $(c,d)\mapsto (f^\prime)_\ast(c)$. 
    The arrow labelled $\beta^\prime$ is defined as $i_\ast-(\phi_2)^\ast$. These are general properties of blow--ups with smooth centers \cite[Proposition 6.7(e)]{F}.
  
  The arrow labelled $\tau$ is defined as $(\phi_1)^\ast(\phi_1)_\ast$. Because of the relation $\phi^\ast=(\phi_1)_\ast(\phi_2)^\ast$, the square
  \begin{equation}\label{simple} \begin{array}[c]{ccc}
    A^j(V) &\xrightarrow{(\phi_1)^\ast}& A^j(Y)\\
    \uparrow{\scriptstyle \phi^\ast}&&\uparrow{\scriptstyle\tau}\\
    A^j(V^\prime)&\xrightarrow{(\phi_2)^\ast}&A^j(Y)\\
    \end{array}\end{equation}
    is commutative. The arrow labelled $\psi$ is defined as $f^\ast f_\ast()\cdot c_1(G)$. The diagram
    \begin{equation}\label{double} \begin{array}[c]{ccc}
        A^{j-1}(E) &\xrightarrow{i_\ast}& A^j(Y)\\
       \ \ \ \ \ \ \ \ \ \ \ \ \ \ \  \uparrow{\scriptstyle f^\ast()\cdot c_1(G)   }&&\ \ \ \  \ \ \uparrow{\scriptstyle (\phi_1)^\ast}\\
        A^{j-2}(Z) &\xrightarrow{(i_Z)_\ast}& A^j(V)\\    
      \ \  \ \  \uparrow{\scriptstyle f_\ast}&&  \ \ \ \ \ \ \uparrow{\scriptstyle (\phi_1)_\ast}\\
        A^{j-1}(E) &\xrightarrow{i_\ast}& A^j(Y)\\
        \end{array}\end{equation}
        is commutative (here $i_Z$ is the inclusion $Z\to V$, and for the upper square we have used \cite[Proposition 6.7(a)]{F}). The commutativity of (\ref{simple}) and (\ref{double}) proves commutativity of diagram (\ref{diag0}).
        
      Since the diagram (\ref{diag0}) is commutative with exact rows, there exists a map $\gamma$ making the diagram
      \begin{equation}\label{diag}      \begin{array}[c]{ccccccc}
     0\to & A^{j-2}(Z) & \xrightarrow{} & A^{j-1}(E)\oplus A^j(V) & \xrightarrow{} & A^j(Y) & \to 0\\
          &\ \ \ \  \uparrow {\scriptstyle \gamma} & & \ \ \ \ \ \  \uparrow  {\scriptstyle (\psi,\phi^\ast)}&& \ \ \ \ \uparrow {\scriptstyle \tau}& \\
          0\to & A^{j-3}(Z^\prime) & \xrightarrow{} & A^{j-1}(E)\oplus A^j(V^\prime) & \xrightarrow{} & A^j(Y) & \to 0\\
          \end{array}\end{equation}
          commute. Applying the snake lemma to diagram (\ref{diag}), we find an exact sequence

  \begin{equation}\label{snake}  \ker\psi\oplus\ker(\phi^\ast)\xrightarrow{g} \ker\tau \to\ \coker\gamma \xrightarrow{(h_1,h_2)} \coker\psi\oplus\coker(\phi^\ast) 
   \xrightarrow{k} \coker\tau \end{equation}

We now state some lemmas about the arrows in (\ref{snake}):

\begin{lemma}\label{l1} The arrow labelled $g$ in (\ref{snake}) is surjective.
\end{lemma}

\begin{proof} Let $c$ be an element in $\ker\tau$, i.e. $c\in A^j(Y)$ with $(\phi_1)^\ast(\phi_1)_\ast(c)=0$. As $(\phi_1)^\ast$ is injective, we must have $(\phi_1)_\ast(c)=0$, and so (as $c$ restricts to $0$ in $A^j(Y\setminus E)$) the element $c$ comes from an element $d\in A^{j-1}(E)$. The element $d$ can be written in a unique way as
  \[ d=d_1+d_2\ \ \in A^{j-1}(E)\ ,\]
  where $d_1=f^\ast(b_1)$ for $b_1\in A^{j-1}(Z)$, and $d_2=f^\ast(b_2)\cdot c_1(G)$ for $b_2\in A^{j-2}(Z)$.
  Using the commutativity of diagram (\ref{double}), we find that
  \[ i_\ast \bigl( f^\ast f_\ast(d)\cdot c_1(G)\bigr) = (\phi_1)^\ast(\phi_1)_\ast (c)=0\ \ \ \hbox{in}\ A^j(Y)\ .\]
  On the other hand, we have
  \[  f^\ast f_\ast(d)\cdot c_1(G) =  f^\ast f_\ast(d_2)\cdot c_1(G)=   f^\ast (b_2)\cdot c_1(G) =  d_2\ \ \ \hbox{in}\ A^{j-1}(E)\ \]
  (here, we have used the splitting property $f_\ast(f^\ast(b_2)\cdot c_1(G))=b_2$ in $A^{j-2}(Z)$ mentioned above), 
  and so 
    \[ i_\ast(d_2)=0\ \ \  \hbox{in}\  A^j(Y)\ .\] 
    Thus, we have equality $c=i_\ast(d_1)$ in $A^j(Y)$ and $d_1\in\ker\psi$, proving the arrow $g$ is surjective.
  \end{proof}

\begin{lemma}\label{l2} The arrow labelled $h_1$ in (\ref{snake}) is $0$.
\end{lemma}

\begin{proof} The map
  \[    A^{j-2}(Z)\ \to\ A^{j-1}(E) \]
  in diagram (\ref{diag}) is defined as $f^\ast()\cdot c_1(G)$.
   By definition of $\psi:=f^\ast f_\ast()\cdot c_1(G)$, the image of $A^{j-2}(Z)$ in $A^{j-1}(E)$ is in the image of $\psi$, and so the arrow $h_1$ is $0$.
\end{proof}

\begin{lemma}\label{l3} The arrow labelled $k$ in (\ref{snake}) is $0$ when restricted to $\coker(\phi^\ast)$.
\end{lemma}

\begin{proof} The map 
  \[ A^j(V)\ \to\ A^j(Y)\]
  in diagram (\ref{diag}) is defined as $(\phi_1)^\ast$. But $(\phi_1)^\ast=(\phi_1)^\ast(\phi_1)_\ast(\phi_1)^\ast\colon A^j(V)\to A^j(Y)$, and so
    \[ \ima\Bigl( A^j(V)\xrightarrow{(\phi_1)^\ast}A^j(Y)\Bigr)\ \subset\ \ima\Bigl( A^j(Y)\xrightarrow{(\phi_1)^\ast(\phi_1)_\ast=:\tau} A^j(Y)\Bigr)\ ,\]
   which shows the arrow $k$ is $0$. 
\end{proof}

Applying lemmas \ref{l1}, \ref{l2}, \ref{l3} to the exact sequence (\ref{snake}), we find that the sequence (\ref{snake}) contains an isomorphism
  \begin{equation}\label{cokeriso}  \coker\bigl( A^{j-3}(Z^\prime)\ \xrightarrow{\gamma}\ A^{j-2}(Z)\bigr)\ \xrightarrow{\cong}\   \coker \bigl( A^j(V^\prime)\ \xrightarrow{\phi^\ast}\ A^j(V)\bigr)\ .\end{equation}

   Let us now determine the cokernel of the map $\gamma$:

\begin{lemma}\label{gamma} There exist isomorphisms
   \[ \begin{split}  A^{j-2}(Z)&\cong A^{j-2}(F)\oplus A^{j-3}(F)\oplus A^{j-4}(F)\ ,\\
                              A^{j-3}(Z^\prime)&\cong A^{j-3}(F)\oplus A^{j-4}(F)\ ,\\
                              \end{split}\]
                 such that the map 
                 \[\gamma\colon A^{j-3}(Z^\prime)\to A^{j-2}(Z)\] 
                (defined by diagram (\ref{diag})) 
        sends $A^{j-3}(F)$ isomorphically to $A^{j-3}(F)$, and $A^{j-4}(F)$ isomorphically to $A^{j-4}(F)$.             
   \end{lemma}
   
   \begin{proof} Since $p^\prime\colon Z^\prime\to F$ is a $\PP^1$--bundle, we can write any $a^\prime\in A^{j-3}(Z^\prime)$ uniquely as
   \[ a^\prime = (p^\prime)^\ast (f_{j-3}) + (p^\prime)^\ast (f_{j-4})\cdot h^\prime\ \ \ \hbox{in}\ A^{j-3}(Z^\prime)\ ,\]
   where $f_{k}\in A^{k}(F)$ and $h^\prime\in A^1(Z^\prime)$ denotes the tautological class. This furnishes the second isomorphism required in lemma \ref{gamma}.
   
   We now consider the image of $a^\prime$ under the induced map
   \[ \gamma\colon\ \ A^{j-3}(Z^\prime)\ \to\ A^{j-2}(Z)\ .\]
   By the above description of the maps in the diagram (\ref{diag}) defining $\gamma$, we have that
   \[ \begin{split}  \gamma(a^\prime)&=  f_\ast \Bigl( f^\ast f_\ast  \bigl(c_2(G^\prime)\cdot (f^\prime)^\ast(a^\prime)\bigr) \cdot c_1(G)  \Bigr)  \\
         &= f_\ast \bigl( c_2(G^\prime)\cdot (f^\prime)^\ast(a^\prime)\bigr)\ \ \ \hbox{in}\ A^{j-2}(Z)\ .
         \end{split} \]
         (Here we have used the splitting property $f_\ast(f^\ast(b)\cdot c_1(G))=b$ mentioned above.)
         
  In particular, a cycle of the form $(p^\prime)^\ast (f_{j-3})$ in $A^{j-3}(Z^\prime)$ is mapped to
   \[ \begin{split} \gamma\bigl((p^\prime)^\ast (f_{j-3})\bigr)&=    f_\ast \Bigl( c_2(G^\prime)\cdot (f^\prime)^\ast (p^\prime)^\ast (f_{j-3}) \Bigr)    \\
                                                                                & =  f_\ast \Bigl( c_2(G^\prime)\cdot f^\ast p^\ast (f_{j-3})\Bigr) \\
                                                                                & =     f_\ast c_2(G^\prime)           \cdot p^\ast (f_{j-3})\ \ \ \hbox{in}\ A^{j-2}(Z)\ .\\
                                                                                \end{split}\]
                                                                     (Here, we have used the ``Moreover'' part of proposition \ref{gs}, plus the projection formula.)
                                                                     
                      Likewise, a cycle of the form $(p^\prime)^\ast (f_{j-4})\cdot h^\prime$ in $A^{j-3}(Z^\prime)$ is mapped to
                      \[ \begin{split} \gamma\bigl((p^\prime)^\ast (f_{j-4})\cdot h^\prime\bigr)&=    f_\ast \Bigl( c_2(G^\prime)\cdot (f^\prime)^\ast(h^\prime)\cdot (f^\prime)^\ast (p^\prime)^\ast 
                                                                                      (f_{j-4}) \Bigr)    \\
                                                                                & =  f_\ast \Bigl( c_2(G^\prime)\cdot (f^\prime)^\ast(h^\prime)\cdot f^\ast p^\ast (f_{j-4})\Bigr) \\
                                                                                & =     f_\ast \bigl(c_2(G^\prime)\cdot (f^\prime)^\ast(h^\prime)\bigr)           \cdot p^\ast (f_{j-4})\ \ \ \hbox{in}\ A^{j-2}(Z)\ .\\
                                                                                \end{split}\]                                                                     
                
              Let us now define
              \[ \begin{split}  h_1&:= f_\ast c_2(G^\prime)\ \ \in A^1(Z)\ ,\\
                                        h_2&:= f_\ast \bigl(c_2(G^\prime)\cdot (f^\prime)^\ast(h^\prime)\bigr)  \ \ \in A^2(Z)\ .\\
                                        \end{split}\]
      By what we have just seen, the map $\gamma\colon A^{j-3}(Z^\prime)\to A^{j-2}(Z)$ verifies
       \begin{equation}\label{desc} \begin{split} \gamma\bigl( (p^\prime)^\ast (f_{j-3})\bigr)&=  h_1\cdot p^\ast(f_{j-3})\ ,\\   
                               \gamma\bigl((p^\prime)^\ast (f_{j-4})\cdot h^\prime\bigr) &= h_2\cdot p^\ast (f_{j-4})\ ,\\
                       \end{split}\end{equation}  
                       and this completely determines the map $\gamma$.      
       The isomorphism 
        \[A^{j-2}(Z)\cong A^{j-2}(F)\oplus A^{j-3}(F)\oplus A^{j-4}(F)\] 
        required in lemma \ref{gamma} is now furnished by the following sublemma:
                        
             \begin{sublemma}\label{gamma2} Any $a\in A^{j-2}(Z)$ can be written uniquely as
             \[ a= p^\ast(f_{j-2})+ h_1\cdot p^\ast(f_{j-3}) + h_2\cdot p^\ast(f_{j-4})\ \ \ \hbox{in}\ A^{j-2}(Z)\ ,\]
             where $f_k\in A^k(F)$.
             \end{sublemma}
             
   \begin{proof} First, we claim that $h_1\in A^1(Z), h_2\in A^2(Z)$ have the following property:
     \begin{equation}\label{proj}  \begin{split} p_\ast(h_2)&= [F]\ \ \ \hbox{in}\ A^0(F)\ ,\\
                            p_\ast(h_1\cdot g)&=[F]\ \ \ \hbox{in}\ A^0(F)\ ,\\
                            \end{split}\end{equation}
                     for some $g\in A^1(Z)$.
                     
           To see this, note that
           \[ \begin{split}  p_\ast(h_2)&= p_\ast f_\ast \bigl( c_2(G^\prime)\cdot (f^\prime)^\ast(h^\prime)\bigr)\\
                                                     &= (p^\prime)_\ast (f^\prime)_\ast \bigl( c_2(G^\prime)\cdot (f^\prime)^\ast(h^\prime)\bigr)\\
                                                     &= (p^\prime)_\ast \bigl( (f^\prime)_\ast c_2(G^\prime)\cdot h^\prime\bigr)\\
                                                     &= (p^\prime)_\ast \bigl( [Z^\prime]\cdot h^\prime\bigr) = (p^\prime)_\ast (h^\prime)=[F]\ \ \ \hbox{in}\ A^0(F)\ .\\
                                                     \end{split}\]
               (Here, the first equality is just the definition of $h_2$; the second equality is the ``Moreover'' part of proposition \ref{gs}; the third equality is the projection formula; the fourth equality is
               the fact that -as noted above- $(p^\prime)_\ast$ is a left--inverse to the arrow $\alpha^\prime$.) This proves the first part of the claim.
               
               For the second part of the claim, let $h$ be the tautological class of the $\PP^2$--bundle $p\colon Z\to F$. We can write
               \[ h_1=c_1 h +p^\ast(d)\ \ \ \hbox{in}\ A^1(Z)\ ,\]
               where $c_1\in\QQ$ and $d\in A^1(F)$. Let us suppose for a moment that $c_1$ were $0$, so $h_1=p^\ast(d)$. 
               Then we would have for any $f_{j-3}\in A^{j-3}(F)$ that
               \[ \gamma \bigl((p^\prime)^\ast (f_{j-3})\bigr) = h_1\cdot p^\ast (f_{j-3}) =p^\ast (f_{j-3}\cdot d)\ \ \ \hbox{in}\ A^{j-2}(Z)\ .\]
               In particular, taking $j=m-1$ we have $f_{m-4}\cdot d=0$ (since $\dim F=m-4$), and so this would imply that
                \[  \gamma \bigl((p^\prime)^\ast A^{m-4}(F)\bigr) =0\ .\]
                Since we know that
                \[ \begin{split} A^{m-4}(F)\oplus A^{m-5}(F)\ &\to\ A^{m-4}(Z^\prime)\ ,\\
                                 (f_{m-4},f_{m-5})\ &\mapsto\ (p^\prime)^\ast(f_{m-4})+ (p^\prime)^\ast(f_{m-5})\cdot h^\prime \\
                              \end{split}\]
                       is an isomorphism, this would imply that  
                 \[ \ima \Bigl( A^{m-4}(Z^\prime)\ \xrightarrow{\gamma}\ A^{m-3}(Z)\Bigr) = \ima\Bigl( A^{m-5}(F)\ \xrightarrow{(p^\prime)^\ast()\cdot h^\prime}\ A^{m-4}(Z^\prime)\ \xrightarrow{\gamma}\ A^{m-3}(Z)\Bigr) \ .\]
             In view of the description of $\gamma$ given in (\ref{desc}), this would imply that
             \[     \ima \Bigl( A^{m-4}(Z^\prime)\ \xrightarrow{\gamma}\ A^{m-3}(Z)\Bigr)\ \subset\  \ima \Bigl( A^{m-5}(F)\ \xrightarrow{ h_2\cdot p^\ast()}\  A^{m-3}(Z)\Bigr)\  .\]
                 But then we would have
                 \[ \coker\Bigl( A^{m-4}(Z^\prime)\ \xrightarrow{\gamma}\ A^{m-3}(Z)\Bigr)\not=0\ \]
               (indeed, the map
                 \[  \begin{split}  A^{m-4}(F)\oplus A^{m-5}(F)\ &\to\     A^{m-3}(Z)  \ ,\\    
                                         (   f_{m-4},f_{m-5})\ &\mapsto\  p^\ast(f_{m-4})\cdot h + p^\ast(f_{m-5})\cdot h_2\\
                                      \end{split}\]
                           is an isomorphism, and so any cycle of the form $p^\ast(f_{m-4})\cdot h$ in $A^{m-3}(Z)$ will be in the cokernel of $\gamma$). In view of the isomorphism
                           (\ref{cokeriso}), this would mean that also
                  \[ \coker\Bigl( A^{m-1}(V^\prime)\ \xrightarrow{\phi^\ast}\ A^{m-1}(V)\Bigr)\not=0\ .\]         
               But this is a contradiction: any curve class on $V$ is represented by a cycle supported on the open $V\setminus Z$ (and likewise on $V^\prime$), and so there is an isomorphism
               $ \phi^\ast\colon A^{m-1}(V^\prime) \xrightarrow{\cong} A^{m-1}(V)$. 
         It follows that $c_1\not=0$ and so 
               \[  p_\ast ( h_1\cdot h) = p_\ast ( c_1 h^2) = c_1 [F]\ \ \ \hbox{in}\ A^0(F)\   .\]   
               Setting $g:= {1\over c_1} h$, this proves the second part of the claim.
               
          Sublemma \ref{gamma2} is now readily proven: it follows from the equalities (\ref{proj}) there are relations
          \begin{equation}\label{rel}  \begin{split}   h &= {1\over c_1} h_1 + p^\ast(d)\ \ \ \hbox{in}\ A^1(Z)\ ,\\
                     h^2 &= h_2 + p^\ast(d_{21})\cdot h + p^\ast(d_{22})\ \ \ \hbox{in}\ A^2(Z)\ ,\\
                     \end{split}\end{equation}
                     for some $d, d_{21}\in A^1(F), d_{22}\in A^2(F)$.
                     
           The projective bundle formula implies that any $a\in A^{j-2}(Z)$ can be written as
           \[   a=  p^\ast(f_{j-2}) + h \cdot p^\ast(f_{j-3}) + h^2\cdot   p^\ast(f_{j-4})\ \ \ \hbox{in}\ A^{j-2}(Z)\ ,\]
           where $f_k\in A^k(F)$. Plugging in the relations (\ref{rel}), we find
           \[   \begin{split} a&=  p^\ast(f_{j-2}) + h \cdot p^\ast(f_{j-3}) + h^2\cdot   p^\ast(f_{j-4})\\
                                 &=  p^\ast(f_{j-2}) + ( {1\over c_1} h_1 + p^\ast(d)) \cdot p^\ast(f_{j-3}) + \Bigl(h_2 + p^\ast(d_{21})\cdot ({1\over c_1} h_1 + p^\ast(d) ) + p^\ast(d_{22})\Bigr)\cdot p^\ast( f_{j-4})\\
                                 &=  p^\ast ( f_{j-2}+ d\cdot f_{j-3}+d\cdot d_{21}\cdot f_{j-4}    +d_{22}\cdot f_{j-4}) +  h_1\cdot  p^\ast({1\over c_1}( f_{j-3} + d_{21}\cdot f_{j-4}))  + h_2\cdot p^\ast(f_{j-4})\\
                                 &= p^\ast(f_{j-2}^\prime) +h_1\cdot p^\ast(f_{j-3}^\prime)+h_2\cdot p^\ast(f_{j-4}^\prime)\ \ \ \hbox{in}\ A^{j-2}(Z)\ ,\\
                                 \end{split}\]
                                 for some $f^\prime_k\in A^k(F)$.
                                 
                          It remains to prove unicity in sublemma \ref{gamma2}: suppose $f_k\in A^k(F)$ is such that
                          \[         p^\ast(f_{j-2})+ h_1\cdot p^\ast(f_{j-3}) + h_2\cdot p^\ast(f_{j-4})=0\ \ \ \hbox{in}\ A^{j-2}(Z)\ .\]
                          Then in particular
                          \[ p_\ast \Bigl(  p^\ast(f_{j-2})+ h_1\cdot p^\ast(f_{j-3}) + h_2\cdot p^\ast(f_{j-4})\Bigr)=0 \ \ \ \hbox{in}\ A^{j-4}(F)\ .\]
                          But the left--hand side equals $p_\ast ( h_2\cdot p^\ast(f_{j-4}))=f_{j-4}$ and so $f_{j-4}=0$. Similarly, the assumption implies
                          \[ p_\ast \Bigl( g\cdot \bigl( p^\ast(f_{j-2})+ h_1\cdot p^\ast(f_{j-3}) \bigr)\Bigr)= p_\ast \bigl( g\cdot h_1\cdot p^\ast(f_{j-3})\bigr)= f_{j-3}   =0 \ \ \ \hbox{in}\ A^{j-3}(F)\ \]
               (where we have used the equality (\ref{proj})). Finally, the assumption implies that 
                             \[ p_\ast \bigl( h_2\cdot p^\ast(f_{j-2})\bigr) = p_\ast(h_2)\cdot f_{j-2} = f_{j-2}=0\ \ \ \hbox{in}\ A^{j-2}(F) \]
                      (where we have used again the equality (\ref{proj})), and so we are done. This proves sublemma \ref{gamma2}, and hence lemma \ref{gamma}.            
                   \end{proof}                                 
                                                                                
   \end{proof}

We are now in position to wrap up the proof of proposition \ref{step1a}. 
Combining the isomorphism (\ref{cokeriso}) and lemma \ref{gamma}, we obtain an isomorphism
  \[ \coker \bigl( A^j(V^\prime)\ \xrightarrow{\phi^\ast}\ A^j(V)\bigr)\ \cong\ \coker \gamma\ \cong\ A^{j-2}(F)\ .\]
  This proves proposition \ref{step1a}. Indeed, it follows from this isomorphism of cokernels there is a commutative diagram with exact rows
   \[ \begin{array}[c]{ccccccc}
      0\to & A^{j-3}(Z^\prime) &\xrightarrow{\gamma} & A^{j-2}(Z) & \xrightarrow{\delta} & A^{j-2}(F) & \to 0\\
             &  \downarrow && \downarrow && \downarrow{\scriptstyle\cong} &\\
      0\to &  A^j(V^\prime) &\xrightarrow{\phi^\ast}& A^j(V) &\to & \coker (\phi^\ast) &\to 0\\     
      \end{array}\]
      As we have seen, the upper row is split exact (lemma \ref{gamma}), and a right--inverse to $\delta$ is given by the pull--back $p^\ast$ (sublemma \ref{gamma2}). It follows the lower row is also split and proposition \ref{step1a} is proven.

  
\end{proof}

The second step of the proof of theorem \ref{main2} consists in extending proposition \ref{step1} to a ``universal isomorphism'' of Chow groups:

\begin{proposition}\label{step2} Let $X\subset\PP^{n+1}(\C)$ be a smooth cubic hypersurface, and let $F=F(X)$ be its Fano variety of lines. Let $M$ be any smooth projective variety.
The natural map
       \[         \begin{split}  A^{j-2}(F\times M)\oplus A^j(W\times M)\ &\to\ A^j(X^{[2]}\times M)\\
                                     (a,b)\ &\mapsto\ (\tau\times\hbox{id}_M)_\ast (p\times \hbox{id}_M)^\ast(a) +(\phi\times\hbox{id}_M)^\ast (b)\\
                                     \end{split}\]
  is an isomorphism for all $j$.  
     \end{proposition}
     
 \begin{proof} For any variety $V$, let $V_M$ denote the product $V\times M$. For a morphism $f\colon X\to Y$, let $f_M\colon X_M\to Y_M$ denote the morphism $f\times{\rm id}_M$.
 Proposition \ref{gs} induces a birational map
   \[  \phi_M:=\phi\times\hbox{id}_M\colon\ \ (X^{[2]})_M\ \dashrightarrow\ W_M\ .\]
 Again using proposition \ref{gs}, we find that the map $\phi_M$ admits a resolution of indeterminacy
  \[ \begin{array}[c]{ccc}
     & Y_M &\\
   {\scriptstyle (\phi_1)_M} \swarrow && \searrow {\scriptstyle (\phi_2)_M}\\
    (X^{[2]})_M\ \ \ \ \ & \xdashrightarrow{\phi_M} & \ \ \ \ \ \ W_M\ .\\
    \end{array}\]  
 Here the morphism $(\phi_1)_M$ is the blow--up with codimension $2$ center $Z_M\subset  (X^{[2]})_M$, and the morphism  $(\phi_2)_M$
 is the blow--up with codimension $3$ center $(Z^\prime)_M\subset W_M$.
 Clearly, the exceptional divisor $E_M\subset Y_M$ fits in a commutative diagram
     \[ \begin{array} [c]{ccc}
  & E_M &\\
  {\scriptstyle {f}_M} \swarrow && \searrow {\scriptstyle (f^\prime)_M}\\
    Z_M\ \ \ \ \ \ & & \ \ \ \ \ \ (Z^\prime)_M\ \\
     {\scriptstyle p_M} \searrow && \swarrow {\scriptstyle (p^\prime)_M}\\
     & F_M & \\
     \end{array}\]
   That is, we are in a set--up where we may apply proposition \ref{step1a} (with $V=(X^{[2]})_M$ and $V^\prime=W_M$) , and so proposition \ref{step2} is proven. 
    \end{proof} 

In the third and final step of the proof of theorem \ref{main2}, we relate $F$ and $X^{[2]}$ on the level of Chow motives. 

\begin{proposition}\label{step3} Let $X\subset\PP^{n+1}(\C)$ be a smooth cubic hypersurface, and let $F=F(X)$ be its Fano variety of lines. 
The map
  \[   \Gamma_\tau\circ {}^t \Gamma_p \oplus {}^t\bar{\Gamma_\phi}  \colon\ \ h(F)(2)\oplus h(W)\ \to\ h(X^{[2]}) \ \ \ \hbox{in}\ \MM_{\rm rat}\]
is an isomorphism.
\end{proposition}

\begin{proof}
   This follows from proposition \ref{step2} by virtue of Manin's identity principle \cite[2.3]{Sch}.   
     \end{proof}

Proposition \ref{step3} proves theorem \ref{main2}, since
  \[ h(W)\cong \bigoplus_{i=0}^n h(X)(i)\ \ \ \hbox{in}\ \MM_{\rm rat}\ \]
  (this is the projective bundle formula for the $\PP^n$--bundle $W\to X$). Theorem \ref{main2} immediately implies theorem \ref{main}: 
 if $X$ has finite--dimensional motive (resp. motive of abelian type), then also $X^{[2]}$ has finite--dimensional motive (resp. motive of abelian type); 
 moreover, the property of having finite--dimensional motive (resp. motive of abelian type) is preserved under taking direct summands.
 \end{proof}

\begin{remark}\label{k0} In \cite[Theorem 5.1]{GS}, proposition \ref{gs} is used to establish a relation between a (not necessarily smooth) cubic $X\subset\PP^{n+1}(k)$ and its Fano variety $F:=F(X)$ in the Grothendieck ring of varieties:
  \[ [X^{[2]}]=[\PP^n][X]+\LLL^2[F]\ \ \ \hbox{in}\ K_0(\hbox{Var/k})\ .\]
Theorem \ref{main2} shows that for smooth cubics over $\C$, a similar relation holds on the level of Chow motives.  
  \end{remark}

\section{Examples}

\begin{corollary}\label{ex} Let $F(X)$ be the Fano variety of lines of a smooth cubic $X\subset\PP^{n+1}(\C)$. In the following cases, $F(X)$ has finite--dimensional motive (of abelian type):

\noindent
(\rom1) $n=3$ or $n=5$;

\noindent
(\rom2) $X$ is a Fermat cubic
  \[  x_0^3+x_1^3+\cdots+x_{n+1}^3=0\ ;\]
  
 \noindent
 (\rom3) $n=4$ and $X$ is defined by an equation
   \[  f(x_0,\ldots,x_3)+x_4^3+x_5^3=0\ ,\]
   where $f(x_0,\ldots,x_3)$ defines a smooth cubic surface;
   
  \noindent
  (\rom4) $n=6$ and $X$ is defined by an equation
  \[ f_1(x_0,\ldots,x_3)+f_2(x_4,\ldots,x_7)=0\ ,\]
  where $f_1, f_2$ define smooth cubic surfaces. 
   
\end{corollary}   

\begin{proof} Appealing to theorem \ref{main}, it suffices to check $X$ has motive of abelian type. In case (\rom2), this is well--known (it follows from the inductive structure of Fermat varieties \cite{Sh}). In case (\rom1), we have
  \[ A^j_{AJ}(X)=0\ \ \hbox{for\ all\ }j \]
  (this is proven in \cite{Lew}, and alternatively in \cite{Ot} and \cite{HI}). This implies the motive of $X$ is generated by curves \cite[Theorem 4]{V2}.
  
  In case (\rom3), the argument is a combination of (\rom1) and (\rom2): Let $X$ be a cubic fourfold as in (\rom3). There is a (Shioda--style) rational map
    \[ \phi\colon\ \ Y\times C\ \dashrightarrow\ X\ ,\]
    where $C$ is a cubic Fermat curve and $Y$ the cubic threefold defined by an equation
    \[  f(x_0,\ldots,x_3)+x_4^3=0\ .\]
    The indeterminacy locus $S$ of $\phi$ is a union of smooth cubic surfaces, and $X$ is dominated by the blow--up of $Y\times C$ with center $S$ (these assertions are proven just as \cite[Theorem 2]{Sh}). This blow--up has motive of abelian type.
    
  The argument for case (\rom4) is similar: there is a (Shioda--style) rational map
  \[ \phi\colon\ \ X_1\times X_2\ \dashrightarrow\ X\ ,\]
  where $X_1, X_2$ are the cubic threefolds defined by the equation
   \[ f_1(x_0,\ldots,x_3)+z_1^3=0\ ,\]
   resp.
   \[ f_2(x_4,\ldots,x_7)+z_2^3=0\ .\]
   The indeterminacy locus $S$ of $\phi$ is a product of two smooth cubic surfaces, and $X$ is dominated by the blow--up of $X_1\times X_2$ with center $S$. Since smooth cubic surfaces and threefolds have motive of abelian type, $X$ has motive of abelian type. 
     \end{proof}
    
\begin{remark} Cubic fourfolds as in corollary \ref{ex}(\rom3) appear in \cite[Example 4.2]{V9}. As shown in loc. cit., to such a fourfold $X$ one can associate a $K3$ surface $S_X$ with the property that there is a correspondence inducing an isomorphism
  \[ A_0^{hom}(S_X)\cong A_1^{alg}(X)\ .\]
  These $K3$ surfaces $S_X$ form a $4$--dimensional family of double covers of $\PP^2$ ramified along a sextic.
    
  The example of corollary \ref{ex}(\rom3) is generalized in \cite{excubic}, where it is shown that smooth cubic fourfolds of type
    \[ f(x_0,\ldots,x_4)+x_5^3=0\ \]
    have finite--dimensional motive. The argument is more involved.
\end{remark}


\begin{acknowledgements} This note is a protracted after--effect of the Strasbourg 2014---2015 groupe de travail based on the monograph \cite{Vo}; thanks to all the participants for the pleasant and stimulating atmosphere. 
Many thanks to Yasuyo, Kai and Len for not being there when I work, and for being there when I don't.
\end{acknowledgements}


\end{document}